\documentclass[12pt,a4paper]{amsart}

\usepackage{latexsym} 
 
\usepackage{graphicx}
\usepackage{epsfig}
\usepackage{mathrsfs}
\usepackage{amssymb}
\usepackage{amsthm}
\usepackage{amsfonts}
\usepackage{amsmath}
\usepackage{centernot}
\usepackage{amstext}
\usepackage{amscd}
\usepackage{enumerate}
\usepackage{tikz}
\usepackage{epic}
\usepackage{eepic}
\usepackage{pstricks,pst-plot}
\usepackage{bigdelim}
\usepackage{multirow}

\numberwithin{equation}{section}

\theoremstyle{plain}
\newtheorem{thm}{Theorem}[section] 
\newtheorem{prop}[thm]{Proposition}
\newtheorem{cor}[thm]{Corollary}
\newtheorem{lem}[thm]{Lemma}

\newtheorem{theorem*}{Theorem}[]

\theoremstyle{definition}
\newtheorem{defn}[thm]{Definition}

\newtheorem{example}[thm]{Example}

\theoremstyle{remark}
\newtheorem{rmk}[thm]{Remark}

\newenvironment{claim}[1]{\par\noindent\underline{Claim:}\space#1}{}
\newenvironment{claimproof}[1]{\par\noindent\underline{Proof:}\space#1}{\hfill $\square$}

\theoremstyle{property}

\newcommand{\N}{\mathbb{N}}
\newcommand{\R}{\mathbb{R}}

\DeclareMathOperator{\grad}{grad\,}

\def\accentclass@{7}
\def\makeacc@#1#2{\def#1{\mathaccent"\accentclass@#2 }}
\makeacc@\cir{017}

\newcommand{\s}{\Sigma}
\newcommand{\vp}{\varphi}

\def\dis{\displaystyle}

\def\det{\mathop{\rm det}}

\def\det{{\text {\rm det}}}

\def\vp{{\varphi}}

\title[Kuo quantity and Thom quantity]
{Equivalence of Kuo and Thom quantities \\
for analytic functions}

\author{Karim Bekka and Satoshi Koike} 

\address{Institut de recherche Mathematique de Rennes, 
Universit\'{e} de Rennes 1, Campus Beaulieu, 35042 Rennes cedex, France}
\address{Department of Mathematics, Hyogo University of Teacher Education,
Kato, Hyogo 673-1494, Japan}
\email{karim.bekka@univ-rennes1.fr}
\email{koike@hyogo-u.ac.jp} 

\date{\today}

\begin{document}

\thanks{This research is partially supported by the Grant-in-Aid 
for Scientific Research (No. 26287011) of Ministry of Education, 
Science and Culture of Japan, and HUTE Short-Term Fellowship
Program 2016.}


\maketitle

\begin{abstract}
Sufficiency of jets is a very important notion introduced 
by Ren\'e Thom in order to establish the structural stability theory.
The criteria for some sufficiency of jets are known 
as the Kuo condition and Thom type inequality, 
which are defined using the Kuo quantity and Thom quantity.
Therefore these quantities are meaningful.
In this paper we show the equivalence of Kuo and Thom quantities. 
Then we apply this result to the relative conditions to a given 
closed set.
\end{abstract}

\bigskip

\section{Introduction}
Let $f : (\R^n ,0) \to (\R,0)$ be a $C^r$ function germ.
The $r$-jet of $f$ at $0 \in \R^n$, $j^r f(0)$, has a unique polynomial
representative $z$ of degree not exceeding $r$.
We do not distinguish the $r$-jet $j^r f(0)$ and the polynomial 
representative $z$ here.

\vspace{3mm}

\noindent {\bf Kuiper-Kuo condition.} There is a positive number 
$C > 0$ such that
\begin{equation*}
\| \grad f(x) \| \ge C \| x\|^{r-1} 
\end{equation*} 
holds in some neighbourhood of $0 \in \R^n.$

\vspace{3mm}

Note that the Kuiper-Kuo condition depends only of its $r$-jet 
$z = j^r f(0)$, and it is independent of the choice of representative.

The Kuiper-Kuo condition is well-known as a criterion for 
$C^0$-sufficiency and $V$-sufficiency of $z$ in $C^r$ functions 
(N. Kuiper \cite{kuiper}, T.-C. Kuo \cite{kuo1}, J. Bochnak and 
S. Lojasiewicz \cite{bochnaklojasiewicz}).
See \S \ref{preli} for the definitions of $C^0$-sufficiency 
and $V$-sufficiency of jet.

Let us recall the Kuo condition.

\vspace{3mm}

\noindent {\bf Kuo condition.} There are positive numbers 
$C, \alpha, \bar w > 0$ such that
\begin{equation*}
\| \grad f(x) \| \ge C \| x\|^{r-1} \text{ in } {\mathcal H}_{r}(f; \bar w) 
\cap \{\| x \| < \alpha\},
\end{equation*} 
where ${\mathcal H}_{r}(f;\bar w) := \{ x \in \mathbb{R}^n : |f(x)| \le 
\bar w \| x\|^{r}\}$ is the {\em horn-neighbourhood of $f^{-1}(0)$ of degree 
$r$ and width $\bar{w}$} (T.-C. Kuo \cite{kuo2}). 

\vspace{3mm}

Note that this Kuo condition is also a condition on the $r$-jet 
$z = j^r f(0)$, and it is independent of the representative $f$.

The Kuo condition is a criterion for $V$-sufficiency of $z$
in $C^r$ functions.

\vspace{3mm}

\noindent {\bf Condition ($\widetilde{K}$).}
There is a positive number $C > 0$ such that
\begin{equation*}
\| x\| \| \grad f(x) \| + |f(x)| \ge C \| x\|^r
\end{equation*} 
holds in some neighbourhood of $0 \in \R^n.$

\vspace{3mm}

This condition is the Kuo condition in a different way.
Therefore condition ($\widetilde{K}$) is also a criterion 
for $V$-sufficiency of $z$ in $C^r$ functions.

On the other hand, R. Thom formulated the following condition 
as a sufficient condition for $z$ to be $C^0$-sufficient in 
$C^r$-functions.

\vspace{3mm}

\noindent {\bf Thom type inequality.} There are positive numbers 
$K, \beta > 0$ such that
\begin{equation*}
\sum_{i<j} \left| x_i \frac{\partial f}{\partial x_j} - 
x_j \frac{\partial f}{\partial x_i} \right|^2 + |f(x)|^2 \geq K \| x\|^{2r}  
\text{ for } \| x\| < \beta.
\end{equation*}

It is shown in \cite{bekkakoike1} that Thom type inequality condition is 
equivalent to the Kuiper-Kuo condition.

Throughout this paper, we denote by $\mathbb{N}$ the set of natural 
numbers in the sense of positive integers.
Let $s \in \N \cup \{ \infty, \omega \}$, and 
let ${\mathcal E}_{[s]}(n,p)$ denote the set of
$C^s$ map-germs : $(\R^n,0)\to (\R^p,0)$.

Now we introduce the Kuo quantity $K_m$ and Thom quantity $T_m.$ 
The Thom quantity is a generalisation of the left side of Thom type 
inequality, and the Kuo quantity is a generalisation of the left side 
of a condition equivalent to condition ($\widetilde{K}$).

\begin{defn}\label{KTquantity}
Let $f\in {\mathcal E}_{[s]}(n,p)$, $n\geq p,$ and let $m \in \N$. 
Let us define two functions of the variable $x$:
\begin{equation}
K_{m}(f,x) := \| x \|^m \sum_{1\leq i_1<\ldots<i_{p}\leq n} 
\left| \det \left( \frac{D(f_1,\ldots, f_p)}{D(x_{i_1},\ldots,x_{i_{p}})}(x)
\right) \right|^m + \| f(x) \|^m
\end{equation}
\begin{equation}
T_{m}(f,x) := \sum_{1\leq i_1<\ldots<i_{p+1}\leq n} 
\left| \det \left(\frac{D(f_1,\ldots, 
f_p,\rho)}{D(x_{i_1},\ldots,x_{i_{p+1}})}(x)\right) \right|^m 
+ \| f(x) \|^m
\end{equation}
where $\rho(x)=\| x\|^2.$
Note that $T_{m}(f,x) = \| f(x) \|^m$ in the case where $n = p$.
\end{defn} 

Related to the Kuo condition and Thom type inequality, we have shown the 
following result.

\begin{thm}(\cite{bekkakoike1})
Let $r \in \N$. 
For $f \in {\mathcal E}_{[r]}(n,p)$, $n \ge p$, the following conditions are 
equivalent. 

\noindent (1) There are positive numbers $C, \alpha > 0$ such that
$K_2(f,x) \ge C \| x \|^{2r}$ for $\| x \| < \alpha .$

\noindent (2) There are positive numbers $K, \beta > 0$ such that 
$T_2(f,x) \ge K \| x \|^{2r}$ for $\| x \| < \beta .$
\end{thm}

The main purpose of this paper is to show the equivalence of 
the Kuo quantity and Thom quantity, which is a generalisation 
of the above result in a certain sense.

\begin{thm}\label{equivKT} (Main Theorem).
Let  $f \in {\mathcal E}_{[\omega ]}(n,p)$, $n \ge p$.
Then for any $m \in \mathbb{N},$ 
$$
K_{m}(f,.)\thickapprox T_{m}(f,.).
$$
\end{thm}

Throughout this paper, we use the equivalence $\thickapprox$ in the following 
sense:

\vspace{2mm}

Let  $f,g : U \to \R$ be non-negative functions,
where $U \subset \R^N$ is an open neighbourhood of $0 \in \R^N$.
If there are real numbers $K > 0$, $\delta > 0$ 
with $B_{\delta}(0) \subset U$ such that
$f(x) \le K g(x)$ for any $x \in B_{\delta}(0)$,
where $B_{\delta}(0)$ is a closed ball in $\R^N$ of radius $\delta$
centred at $0 \in \R^N$, 
then we write $f \precsim g$ (or $g \succsim f$).
If $f \precsim g$ and $f \succsim g$, we write $f \thickapprox g$.

\vspace{2mm}

In the next section we mention the definitions of $C^0$-sufficiency and 
$V$-sufficiency of jets, and give the notion of the relative jet 
of a $C^s$ mapping to a given closed set $\s$.
We shall show our Main Theorem in \S \ref{proof}, and apply 
the theorem to the relative conditions to a closed set $\s$ 
in \S \ref{application}.


\section{Preliminaries}\label{preli}

\subsection{Sufficiency of jets}\label{sufficiency}
Let $s \in \N \cup \{ \infty, \omega \}$.
Let us recall ${\mathcal E}_{[s]}(n,p)$, the set of $C^s$ map-germs 
: $(\R^n,0)\to (\R^p,0)$.
Let $j^r f(0)$ denote the r-jet ($r \in \N$) of $f$ at $0 \in \R^n$ for 
$f \in {\mathcal E}_{[s]}(n,p)$, $s \ge r$, and let $J^r(n,p)$ denote 
the set of r-jets in ${\mathcal E}_{[s]}(n,p)$.

We say that $f,g\,\in {\mathcal E}_{[s]}(n,p)$ are $C^0$-{\em equivalent} 
(resp. $SV$-{\em equivalent}), if there exists a local homeomorphism 
$\sigma : (\R^n,0) \to (\R^n,0)$ such that $f = g \circ \sigma$ 
(resp. $\sigma (f^{-1}(0)) = g^{-1}(0)$).
In addition, we say that  $f,g\,\in {\mathcal E}_{[s]}(n,p)$ are 
$V$-{\em equivalent}, if $f^{-1}(0)$ is homeomorphic to $g^{-1}(0)$
as germs at $0\in \mathbb{R}^n$.

Let $w \in J^r(n,p).$ 
We call the $r$-jet $w$ $C^0$-{\em sufficient}, $SV$-{\em sufficient} and 
$V$-{\em sufficient} in $C^s$ mappings, $s \ge r$, if any two realisations
$f$, $g\,\in {\mathcal E}_{[s]}(n,p)$ of $w,$ 
namely $j^rf(0) = j^rg(0)=w,$  are $C^0$-equivalent, $SV$-equivalent
and $V$-equivalent, respectively.

Let us recall the Thom type inequality for $f \in {\mathcal E}_{[s]}(n,p)$, 
$n \ge p$ \ :

\vspace{2mm}

\noindent There are positive numbers $K, \alpha, \beta > 0$ such that 
$T_2(f,x) \ge K \| x \|^{\alpha}$ for $\| x \| < \beta .$

\vspace{2mm}

\noindent As mentioned in the Introduction, R. Thom considered 
this condition with $\alpha = 2r$ in the function case 
as a sufficient condition for $z = j^r f(0)$ to be $C^0$-sufficient 
in $C^r$ functions. 
On the other hand, he considered this condition in the mapping case 
as a sufficient condition for $SV$-sufficiency of jet. 

The Kuo condition mentioned in the Introduction is a criterion 
for $V$-sufficiency of $z = j^r f(0)$ in $C^r$ functions.
This condition is generalised to the mapping case,
as a criterion for $V$-sufficiency of $z = j^r f(0)$ in $C^r$ mappings : 
$(\R^n,0) \to (\R^p,0)$, $n \ge p$.
For the details, see T.-C. Kuo \cite{kuo3}.

\subsection{Relative jet to a given closed set}\label{relativejet}
Throughout this paper, let $\s$ be a germ of a given closed subset of $ \R^n$ 
at $0 \in \R^n$ such that $0 \in \s.$ 
Then we denote by $d(x,\s)$ the distance from a point $x \in \R^n$ 
to the subset $\s.$
 
We consider on ${\mathcal E}_{[s]}(n,p)$ the following equivalence relation: 

\vspace{1mm}

\noindent Two map-germs  $f,g\,\in {\mathcal E}_{[s]}(n,p)$ are 
$r$-$\Sigma$-{\em equivalent}, denoted by $f\sim g$, if there exists 
a neighbourhood $U$ of $0$ in $\R^n$ such that the r-jet extensions of 
$f$ and $g$ satisfy $j^rf(\s\cap U)= j^rg(\s\cap U).$

\vspace{1mm}

\noindent We denote by $j^rf(\s;0)$ the equivalence 
class of $f,$ and by $J^r_{\s}(n,p)$ the quotient set 
${\mathcal E}_{[s]}(n,p)/\sim.$ 

We can define the notions of $C^0$-sufficiency, $SV$-sufficiency and 
$V$-sufficiency of relative jets to $\s$, similarly to in the non-relative 
case. 
In \cite{bekkakoike2} we gave criteria for the relative $r$-jet to be 
$C^0$-sufficient and $V$-sufficient in ${\mathcal E}_{[r]}(n,p)$ 
or ${\mathcal E}_{[r+1]}(n,p)$, using the relative Kuiper-Kuo condition 
and relative Kuo condition (or condition ($\widetilde{K}_{\s}$)), 
respectively.


\section{Proof of Main Theorem}\label{proof}

In this section we show the equivalence between the Kuo quantity $K_{m}$ 
and the Thom quantity $T_{m}$, namely our main theorem (Theorem 
\ref{equivKT}).
Before we give the proof, let us examine an example.

\begin{example}
Let $f = (f_1, f_2) : (\R^2,0) \to (\R^2,0)$ be a polynomial mapping
defined by 
$\dis 
f_1(x,y) = x - y^2, \ \  f_2(x,y) = x^2.
$
Then we have
$
f_1(x,y)^2 + f_2(x,y)^2 = (x - y^2)^2 + x^4, 
$
$\dis
\det \left( \frac{D(f_1, f_2)}{D(x,y)}((x,y))\right) = 4xy.
$
Therefore we have
$$
T_{2}(f,(x,y)) = (x - y^2)^2 + x^4, \ \ 
K_{2}(f,(x,y)) = 16(x^2 + y^2)x^2y^2 + (x - y^2)^2 + x^4.
$$

To show that $T_{2}(f,(x,y)) \thickapprox K_{2}(f,(x,y)),$ we consider two 
cases.

In the case where $|x - y^2| \le \frac{1}{2} y^2$, we have  
$x\ge  \frac{1}{2} y^2$.
Therefore 
$64 x^4\geq 16x^2y^4$ and  since for any constant $C> 65,$ $16x^4y^2=o((C-65)x^4)$
 we get 
$$ 
C T_{2}(f,(x,y)) \ge K_{2}(f,(x,y)) \ge T_{2}(f,(x,y))
$$
in a small neighbourhood of $(0,0) \in \mathbb{R}^2,$

In the case where $|x - y^2| \ge \frac{1}{2} y^2$
we can see that
$$
(x - y^2)^2 + x^4 \ge \frac{1}{4}y^4 + x^4 \ge
16x^2y^4 + 16x^4y^4 = 16(x^2 + y^2)x^2y^2
$$
in a small neighbourhood of $(0,0) \in \R^2$.

Thus, for any constant $C> 65,$ we have 
$$
T_{2}(f,(x,y)) \le K_{2}(f,(x,y)) \le C T_{2}(f,(x,y))
$$
in a small neighbourhood of $(0,0) \in \R^2,$
it follows that $T_{2}(f,(x,y)) \thickapprox K_{2}(f,(x,y)).$
\end{example}

Let $ord(\gamma(t))$ denote the order of $\gamma$ in $t$ for
a $C^\omega$ function $\gamma: [0, \delta)\to \mathbb{R}.$
	
\begin{proof}[Proof of Theorem \ref{equivKT}] 
It is obvious that $K_{m}(f,.)\succsim T_{m}(f,.).$ 
Therefore we have to show the converse.

We first remark that if $x$ and $y$ are bigger than or equal to $0$, we have 
$$
(x+y)^m \geq x^m + y^m \geq \frac{(x+y)^m}{2^m}.
$$
It follows that
$$
K_{m}(f,x)\thickapprox v^{m}(x)+u^{m}(x)\thickapprox (h(x))^m$$
$$ 
T_{m}(f,x)\thickapprox w^{m}(x)+u^{m}(x)\thickapprox (g(x))^m,
$$  
where $u(x)=\Vert f(x) \Vert$, 

\vspace{3mm}

\qquad $v(x)=\Vert x\Vert \dis \sum_{1\leq i_1<\ldots<i_{p}\leq n} 
\left| \det \left( \frac{D(f_1,\ldots, f_p)}{D(x_{i_1},\ldots,
x_{i_{p}})}(x)\right) \right| , $

\vspace{3mm}

\qquad $w(x)=\dis \sum_{1\leq i_1<\ldots<i_{p+1}\leq n} 
\left| \det \left( \frac{D(f_1,\ldots, f_p,\rho)}{D(x_{i_1},\ldots,
x_{i_{p+1}})}(x)\right) \right|, (\text{where } \rho(x)=\| x\|^2),$ 

\vspace{3mm}

\noindent $h(x)=v(x)+u(x)$ and $g(x)=w(x)+u(x).$

Suppose now that  $K_{m}(f,.)\precsim T_{m}(f,.)$  does not hold.  
Then by the curve selection lemma, there is a $C^\omega$ curve
$\tilde\lambda=(\lambda, C): [0, \delta) \to \mathbb{R}^n\times \mathbb{R}$ 
with 
$\tilde\lambda(0)=(0,0)$ and $\tilde\lambda(t)\in 
(\mathbb{R}^n \setminus \{0\}) \times \mathbb{R}^*,$ for $t\ne 0,$
such that
\begin{equation}\label{eq1}
(C(t))^m K_{m}(f,\lambda(t))>T_{m}(f,\lambda(t)).
\end{equation}
We may write \eqref{eq1} as:
\begin{equation}\label{eq2}
(C(t)(h \circ \lambda (t)))^m>(g \circ \lambda (t))^m .
\end{equation}
Here we remark that the functions $g \circ \lambda , h \circ \lambda ,
u \circ \lambda ,v \circ \lambda$ and $w \circ \lambda$  are real analytic on
$[0, \delta)$ and satisfying the conditions
$$
g \circ \lambda (0)=h \circ \lambda (0)=u \circ \lambda (0)=v \circ \lambda (0)
=w \circ \lambda (0)=0
$$
and
$$
\lambda(t)\neq0,\quad  C(t)>0,\quad h \circ \lambda (t)>0,\quad
g \circ \lambda (t)\geq0\quad\textrm{for}\quad 0<t<\delta
$$

By \eqref{eq2}, $C(t)(h \circ \lambda (t))>u \circ \lambda (t)$,
$C(t)(h \circ \lambda (t))>w \circ \lambda (t)$ and 
$$
v \circ \lambda (t)=h \circ \lambda (t)
-u \circ \lambda(t)\geq h \circ \lambda (t)(1-C(t)).
$$
Then we have
\begin{equation}
\begin{cases}
ord(C)+ord(h \circ \lambda )\leq ord(u \circ \lambda )\\
ord(C)+ord(h \circ \lambda )\leq ord(w \circ \lambda )\\
ord(v \circ \lambda )\leq ord(h \circ \lambda ).
\end{cases}
\label{eq3}
\end{equation}
Note that we are not considering the second inequality in the case 
where $n = p$.

Let $\tilde\lambda$ be written as follows
$\lambda_i(t)= a_1^{(i)}t^{\varepsilon_1(i)}+a_2^{(i)}t^{\varepsilon_2(i)}+ \ldots$

\vspace{3mm}

\noindent where $1\leq {\varepsilon_1(i)}<{\varepsilon_2(i)}< 
\ldots$ and $\left \{\begin{array}{clcr}
a_1^{(i)} \ne 0& if & \lambda_i(t)\not\equiv 0\\
{\varepsilon_1(i)}=\infty& if & \lambda_i(t)\equiv 0\end{array}\right.$ 
$(1\leq i\leq n),$

\vspace{3mm}

\noindent $C(t)= u_1t^{b_1}+u_2t^{b_2}+ \ldots$
where $1\leq b_1<b_2<\ldots$ and $u_1\ne 0.$

Since condition \eqref{eq1} is invariant under rotation, we can assume that $\varepsilon_1(1)<\varepsilon_1(i)$ for $i\ne 1.$

Let $f_j(\lambda(t))= d_1^{(j)}t^{q_1^{(j)}}+d_2^{(j)}t^{q_2^{(j)}}+ \ldots$, 
where $1\leq q_1^{(j)}<q_2^{(j)}< \ldots$  $(1\leq j\leq p).$
Then $$\frac{df_j \circ\lambda}{dt}(t)= q_1^{(j)}d_1^{(j)}t^{q_1^{(j)}-1}+q_2^{(j)}d_2^{(j)}t^{q_2^{(j)}-1}+ \ldots
\qquad(1\leq j\leq p).$$
It follows from \eqref{eq3} that
\begin{equation}\label{eq4}
q_1^{(j)}\geq  ord(C)+ord(h \circ \lambda )\qquad \text{ 
for all } j\in \{1,\ldots, p\}.
\end{equation}

By \eqref{eq3} again, we have

\begin{equation}\label{eq5}
\varepsilon_1(1) +ord\left(\sum_{1\leq i_1<\ldots<i_p\leq n} \left|
\det \left( \frac{D(f_1,\ldots, f_p)}{D(x_{i_1},\ldots, 
x_{i_p})}(\lambda(t))\right) \right|\right)\leq ord(h \circ \lambda ).
\end{equation}

Therefore there is a $p$-tuple of integers $(k_1,\dots, k_p)$ 
with $1\leq k_1<\dots<k_p\leq n$ such that

\begin{equation}\label{eq6}
\left \{\begin{array}{clcr}
ord(\vert \det \left( \frac{D(f_1,\ldots, f_p)}{D(x_{k_1},\ldots, 
x_{k_p})}(\lambda(t))\right) \vert) \leq 
ord(\vert \det \left( \frac{D(f_1,\ldots, f_p)}{D(x_{i_1},\ldots, 
x_{i_p})}(\lambda(t))\right) \vert)\, \\
\text{for any}\,(i_1,\ldots, i_p), \ \ \text{and}\\ 
\qquad ord(\vert \det \left( \frac{D(f_1,\ldots, f_p)}{D(x_{k_1},
\ldots, x_{k_p})}(\lambda(t))\right) \vert) \leq 
ord(h\circ \lambda ) -\varepsilon_1(1). 
\end{array}\right.
\end{equation}

We continue the proof of the converse, dividing it into two cases.
We first consider the case where $n > p$.
Then we have the following.

\vspace{3mm}

\begin{claim} 
$k_{1}>1.$
\end{claim}

\begin{claimproof}
Since 
$\dis\frac{df_j\circ\lambda}{dt}(t)= \sum _{i=1}^n{\frac{\partial f_j}{\partial x_i}}
(\lambda(t)){\frac{d\lambda_i}{dt}(t)},
 \ (1\leq j\leq p),
$
we have
 
 \begin{equation}
 \left (\begin{array}{clcr}
\frac{df_1\circ\lambda}{dt}(t)\\
\vdots\\
\frac{df_p\circ\lambda}{dt}(t) \end{array}\right)
= {\frac{d\lambda_1}{dt}}(t)
\left (\begin{array}{clcr}
{\frac{\partial f_1}{\partial x_1}}(\lambda(t))\\
\vdots\\
{\frac{\partial f_p}{\partial x_1}}(\lambda(t)) \end{array}\right)
+\ldots+ {\frac{d\lambda_n}{dt}}(t)\left (\begin{array}{clcr}
{\frac{\partial f_1}{\partial x_n}}(\lambda(t))\\
\vdots\\
{\frac{\partial f_p}{\partial x_n}}(\lambda(t)) \end{array}\right).
\label{eq7}
 \end{equation}
 
\noindent Here we remark that, by \eqref{eq4}

\begin{equation}\label{eq8}
ord\left(\dis \frac{1}{\lambda'_1(t)}.\frac{df_j\circ\lambda}{dt}(t)\right) =q_1^{(j)}-\varepsilon_1(1)  \geq ord(C)+ord(h\circ \lambda)-\varepsilon_1(1) 
\quad (1\leq j\leq p),
\end{equation}
and 
\begin{equation}\label{eq9}
ord\left(\frac{\lambda'_i(t)}{\lambda'_1(t)}\right) \geq1 
\qquad (2\leq i\leq n) .
\end{equation}

Assume, by contradiction, that $k_{1}=1$ in \eqref{eq6}. For simplicity, set 
$$
A(t)=\left(\begin{array}{cccc}{\frac{\partial f_1}{\partial x_1}}(\lambda(t))&
{\frac{\partial f_1}{\partial x_{k_2}}}(\lambda(t))&\ldots&
{\frac{\partial f_1}{\partial x_{k_p}}}(\lambda(t))\\
\vdots & \vdots & & \vdots \\
{\frac{\partial f_p}{\partial x_1}}(\lambda(t))&
{\frac{\partial f_p}{\partial x_{k_2}}}(\lambda(t))&\ldots&
{\frac{\partial f_p}{\partial x_{k_p}}}(\lambda(t))\end{array}\right).
$$
Then the determinant of the matrix $A(t)$ is the summation of determinants of
the following matrices:

\begin{equation}
\begin{pmatrix}
\ldelim({3}{1cm}[$\frac{1}{\lambda'_1(t)}$]&\frac{df_1\circ\lambda}{dt}(t)&\rdelim){3}{0.4cm}[]
&{\frac{\partial f_1}{\partial x_{k_2}}}(\lambda(t))&\ldots&{\frac{\partial f_1}{\partial x_{k_p}}}(\lambda(t))\\
&\vdots&&\vdots&&\vdots\\
&\frac{df_p\circ\lambda}{dt}(t)&&{\frac{\partial f_p}{\partial x_{k_2}}}(\lambda(t))&\ldots&{\frac{\partial f_p}{\partial x_{k_p}}}(\lambda(t))
\end{pmatrix}
\label{eq10}
\end{equation}

\begin{equation}
\begin{pmatrix}
\ldelim({3}{1.5cm}[$-\frac{\lambda'_i(t)}{\lambda'_1(t)}$]&
{\frac{\partial f_1}{\partial x_i}}(\lambda(t))&\rdelim){3}{0.4cm}[]
&{\frac{\partial f_1}{\partial x_{k_2}}}(\lambda(t))&\ldots&{\frac{\partial f_1}{\partial x_{k_p}}}(\lambda(t))\\
&\vdots&&\vdots&&\vdots\\
&{\frac{\partial f_p}{\partial x_i}}(\lambda(t))&&{\frac{\partial f_p}{\partial x_{k_2}}}(\lambda(t))&\ldots&{\frac{\partial f_p}{\partial x_{k_p}}}(\lambda(t))
\end{pmatrix}
 \, \text{for } i\in\{2,\ldots, n\}. 
 \label{eq11}
\end{equation}

By \eqref{eq8} the order of the determinant of the matrix \eqref{eq10} is  
bigger than or equal to $ord(C)+ord(h\circ \lambda)-\varepsilon_1(1)$,
and  by \eqref{eq5} the order of the determinant of the matrix \eqref{eq11} is 
bigger than the order of the determinant of the matrix \eqref{eq10}.
Therefore we have  
$$
ord(\vert \det A(t)\vert)\geq ord(C)+ord(h \circ \lambda )-\varepsilon_1(1)>ord(h\circ \lambda)-\varepsilon_1(1)
$$
which contradicts $\eqref{eq6}$. 
This completes the proof of the claim.
\end{claimproof} 

\vspace{3mm}

It follows from the Claim that there is a $p$-tuple $(k_1,\dots, k_p)$
with $1<k_1<\dots< k_p\leq n$ such that condition  \eqref{eq6} holds.
Then 
$$
ord\left(\left| \det \left( \frac{D(f_1,\ldots, f_p,\rho)}{D(x_{1},x_{k_1},
\ldots, x_{k_p})}(\lambda(t))\right) \right| \right) \leq
ord(\lambda) + ord\left(\left| \det \left( \frac{D(f_1,\ldots, f_p)}{D(x_{k_1},
\ldots, x_{k_p})}(\lambda(t)) \right) \right|\right)
$$
$$
\leq \varepsilon_1(1)+ord(h \circ \lambda ) -\varepsilon_1(1) 
= ord(h \circ \lambda ).
$$ 
This contradicts \eqref{eq3}, and it follows that
$K_{m}(f,.)\precsim T_{m}(f,.)$.

We next consider the case where $n = p$.
Using a similar argument to the proof of the above Claim,
we get the same contradiction for
$$
A(t)=\left(\begin{array}{cccc}{\frac{\partial f_1}{\partial x_1}}(\lambda(t))&
{\frac{\partial f_1}{\partial x_2}}(\lambda(t))&\ldots&
{\frac{\partial f_1}{\partial x_n}}(\lambda(t))\\
\vdots & \vdots & & \vdots \\
{\frac{\partial f_n}{\partial x_1}}(\lambda(t))&
{\frac{\partial f_n}{\partial x_2}}(\lambda(t))&\ldots&
{\frac{\partial f_n}{\partial x_n}}(\lambda(t))\end{array}\right).
$$
Therefore it follows that $K_{m}(f,.)\precsim T_{m}(f,.)$,
and this completes the proof.
\end{proof}

\begin{rmk} The proof of Theorem \ref{equivKT} uses essentially the curve 
selection lemma. 
Therefore it is not difficult to see that the results are still valid if we 
suppose only that $f$ is an arc-analytic and differentiable subanalytic 
map-germ; see \cite{kurdyka}, \cite{hironaka} and \cite{bierstonemilman} 
for the notions and properties of subanalytic and arc-analytic functions.

\end{rmk}



\section{Applications of our main result to the relative case}\label{application}


\subsection{$\s$-$r$-compatibility}

We now introduce some notion for a $C^r$-map germ 
$f:(\mathbb{R}^n,0)\to (\mathbb{R}^p,0)$ in order to extend to the relative 
case the previous equivalence defined in the non-relative case.

Let $\s$ be a germ at $0 \in \R^n$ of closed set such that $0 \in \s$.
Given a map $g\in {\mathcal E}_{[r]}(n,p)$ with $j^rg(\s ;0) = j^rf(\s ;0)$, 
let $f_t: (\mathbb{R}^n,0)\to (\mathbb{R}^p,0)$ denote the $C^r$ mapping 
defined by $$f_t(x)= f(x) + t(g(x)-f(x)) \text{ for }\, \,t\in [0,1].$$

\begin{defn}\label{rcompatibility}
A condition $(*)$ on a $C^r$ map $f$ is called $\s$-$r$-{\em compatible 
in the direction} $g$, if $f_t$ satisfies
condition $(*)$ for any $t\in [0,1].$
If condition $(*)$ is $\s$-$r$-compatible in any direction $g\in {\mathcal E}_{[r]}(n,p)$ with $j^rg(\s ;0) = j^rf(\s ;0)$,
we simply say condition $(*)$ is $\s$-$r$-{\em compatible}.
\end{defn}
Let $f: ( \mathbb{R}^n,0)\to (\mathbb{R}^p,0)$ be a $C^1$ map-germ, 
$\s\subset \mathbb{R}^n$ be a germ of a closed set such that $0\in \s$ and 
$r\in \mathbb{N}.$
For $m\in \mathbb{N}$, we introduce the following conditions:
\begin{gather*}
I^T_{r}(m):\ \exists c,\delta>0\, \text{ such that }  T_{m}(f,x)\geq c(d(x,\s))^{rm}
\end{gather*}
  for $ \|x\|<\delta ,$
\begin{gather*}
I^K_{r}(m):\ \exists c,\delta>0\, \text{ such that } 
K_{m}(f,x)\geq c(d(x,\s))^{rm}
\end{gather*}
for $\|x\|<\delta .$
\begin{rmk} If $f$ is $C^{\omega},$ we have from Theorem \ref{equivKT}, 
for any $m\in \mathbb{N}$,
$$
I^T_{r}(m) \ \text{holds if and only if} \ I^K_{r}(m) \ \text{holds.}
$$
\end{rmk}

\begin{prop}\label{Propcomp} 
The conditions  $I^T_{r}(m)$ and $I^K_{r}(m)$ are r-compatibles.
\end{prop} 

\begin{proof}
Let $f_{t}=f+th$ with $h = g - f.$
Then $\|h\|=o(d(.,\s)^{r}),$  $\|f_{t}\|\geq \|f\|-\|h\|$ and the expansion of 
the determinants give
$$
T_{m}(f_{t},x)= T_{m}(f,x)+o(d(x,\s))^{rm}
$$
 and 
$$
K_{m}(f_{t},x)= K_{m}(f,x)+o(d(x,\s))^{rm}.
$$
Thus the r-compatibilities of $I^T_{r}(m)$ and $I^K_{r}(m)$ follow.
\end{proof} 

As a corollary of Theorem \ref{equivKT}, we have the following result. 

\begin{cor}\label{corollary1} 
Let $\s$ be a germ at $0 \in \R^n$ of a closed set such that 
$0 \in \s$.
Let $r \in \N$, and let $f\in {\mathcal E}_{[r]}(n,p) $, $n\geq p$.  
Suppose that $j^rf(\s,0)$ has a $C^\omega$ realisation. \\
Then for any $m\in \mathbb{N}$, 
$$
I^T_{r}(m) \text{  holds if and only if  }  I^K_{r}(m) \text{   holds.  }
$$ 
\end{cor}

\begin{proof}
Let $g : (\R^n,0) \to (\R^p,0)$ be a $C^\omega$ realisation of
$j^rf(\s,0)$.
From Theorem \ref{equivKT}, conditions $I^T_{r}(m)$ and $I^K_{r}(m)$ are
equivalent for $g$.
Now, by Proposition \ref{Propcomp},  the result follows. 
\end{proof} 

\begin{rmk} 
As pointed out in \cite{bekkakoike2}, any $r$-jet, $r \in \N$, has 
a unique polynomial realisation of degree not exceeding $r$ in the 
non-relative case, but some $r$-jets do not have even a $C^{\omega}$ 
realisation in the general relative case. 
Therefore, in the above theorem, the assumption that $j^rf(\s,0)$ has 
a $C^\omega$ realisation makes sense.
\end{rmk} 

\begin{lem}\label{lemma1}
For $X_{1},\ldots,X_{l}\geq 0$  and a positive integer 
$m \in \mathbb{N}$,
we have
$$
(X_{1}+\ldots+X_{l})^m \thickapprox X_{1}^m+\ldots+X_{l}^m. 
$$ 
Therefore we see that $K_{1}\thickapprox T_{1}$ if and only if for any  
$m\in \mathbb{N},\,\, K_{m}\thickapprox T_{m} .$
\end{lem}

As a corollary of Lemma \ref{lemma1} and Corollary \ref{corollary1}, 
we have the following result.

\begin{cor}
Let $\s$ be a germ at $0$ of a closed set.
Let $r \in \N$, and let $f\in {\mathcal E}_{[r]}(n,p) $, $n\geq p$.
Suppose that $j^rf(\s,0)$ has a $C^\omega$ realisation.
Then the following conditions are equivalent:

\begin{enumerate} [(1)] 
\item There exists $m\in \mathbb{N}$ such that $I^T_{r}(m)$ holds
\item For all $m\in \mathbb{N}$, $I^T_{r}(m)$ holds
\item  There exists $m\in \mathbb{N}$ such that $I^K_{r}(m)$ holds
\item For all $m\in \mathbb{N}$, $I^K_{r}(m)$ holds
\end{enumerate}  
\end{cor}
\begin{rmk}
It follows from the proof of Theorem \ref{equivKT} that the equivalence 
between  conditions $T_{m}$ and $K_{m}$ holds for any $ C^1$ map $f$ 
in a category where the analytic curve selection lemma is valid.
\end{rmk}


\subsection{Characterisations of finite $\s$-$SV$-determinacy} 

Let $\mathcal{E}(n)^p$, $n > p$, be the set of $C^{\infty}$ 
map-germs : $\R^n \to \R^p$ at $0 \in \R^n$, and let $\s$ be
a germ of closed subset of $\mathbb{R}^n$ such that $0\in \s.$
We say that $f,g\,\in {\mathcal E}_{[s]}(n,p)$ are $\s$-SV-{\em equivalent} 
if there is a local homeomorphism $\vp : (\R^n,0) \to (\R^n,0)$ fixing $\s$ 
such that $\vp (f^{-1}(0))=g^{-1}(0)$, and 
$f,g\,\in {\mathcal E}_{[s]}(n,p)$ are 
$\s$-V-{\em equivalent} if $f^{-1}(0)$ is homeomorphic to $g^{-1}(0)$ 
as germs at $0\in \mathbb{R}^n$ by a homeomorphism which fixes 
$f^{-1}(0)\cap \s$. 
Then $f \in \mathcal{E}(n)^p$ is called {\em finitely} 
$\s$-$SV$-{\em determined} (resp. {\em finitely} $\s$-$V$-{\em determined}) 
if there is a positive integer $k$
such that for any $g \in \mathcal{E}(n)^p$ with $j^k g(\s ;0) = 
j^k f(\s ;0),$ $g$ is $\s$-$SV$-equivalent (resp. $\s$-$V$-equivalent) to $f$.

Let $\varphi=(\varphi_{1},\ldots,\varphi_{p}): \R^n \to \R^p$, $n > p$,
be a $C^{\infty}$ map-germ at $0 \in \R^n$.
We denote by $I_{K}(\varphi)$ the ideal of ${\mathcal E}(n)$ generated by 
$ \varphi_1,\ldots, \varphi_p$ and the Jacobian determinants
$$
\dis \frac{D(\varphi_1,\ldots, \varphi_p)}{D(x_{i_1},\ldots, x_{i_p})}(x),
\ (1\leq i_1<\ldots<i_p\leq n). 
$$
We define also the ideal of ${\mathcal E}(n)$, 
denoted by $I_{T}(\varphi)$, generated by $ \varphi_1,\ldots, \varphi_p$ and 
the Jacobian determinants
$$
\dis \frac{D(\varphi_1,\ldots, \varphi_p,\rho)}
{D(x_{i_1},\ldots,x_{i_{p+1}})}(x), \ (1\leq i_1<\ldots<i_{p+1}\leq n)
$$
where $\rho(x)=\|x\|^2.$

Let ${\frak m}_{\Sigma}^\infty$ be the 
ideal of ${\mathcal E}(n)$ consisting of germs $f$ such that $j^\infty f(x)=0$ for all $x\in \Sigma$, namely 
$\dis 
{\frak m}_{\Sigma}^\infty = \{f\in {\mathcal E}(n):\, j^{\infty} f(\s ;0)=0\}.
$

\begin{defn} Let $I$ be an ideal of $ \mathcal{E}(n).$
We say that $I$ is {\it $\s$-elliptic} if there is $f\in I$ such that 
$\dis |f(x)|\geq Cd(x,\s)^\alpha$
in a neighbourhood of $0$, where $C$ and $\alpha$ are positive constants.
\end{defn} 

A germ of closed subset $\s$  of $ \mathbb{R}^n$ is called {\it coherent} if ${\frak m}_{\Sigma}$ is 
a finitely generated ideal of $ \mathcal{E}(n).$
Then we have the following characterisations of finite $\s$-$SV$-determinacy.

\begin{thm}\label{T1}(\cite{bekkakoike2}) Let $\s$ be a coherent  germ of closed subset 
of $\mathbb{R}^n$ such that $0\in \s.$
Then the following conditions are equivalent for 
$\varphi \in \mathcal{E}(n)^p$ where $n> p$:

\begin{enumerate}[(1)]
\item $\varphi$ is finitely $\Sigma$-SV-determined.
\item $\varphi$ is finitely $\Sigma$-V-determined.
\item  $I_{K}(\varphi)$ is $\s$-elliptic.
\item ${\frak m}_{\Sigma}^\infty \subset I_{K}(\varphi).$\\
If moreover $\varphi$ is analytic, they are also equivalent to:
\item ${\frak m}_{\Sigma}^\infty \subset I_{T}(\varphi).$

\end{enumerate}
\end{thm}

Theorem \ref{equivKT} takes an important role in the proof of  the above theorem. 
For the detailed proof and more characterisations of finite $\s$-$SV$-determinacy, 
see \S 5 in \cite{bekkakoike2}.


\end{document}